\documentclass[a4paper,12pt]{amsart}

\usepackage{amsmath}
\usepackage{amssymb}
\usepackage{ifthen}
\usepackage{graphicx}
\usepackage[T1]{fontenc} 

\nonstopmode \numberwithin{equation}{section}
\newtheorem{thm}[equation]{Theorem}
\newtheorem{cor}[equation]{Corollary}
\newtheorem{lem}[equation]{Lemma}
\newtheorem{prop}[equation]{Proposition}

\newtheorem{conj}{Conjecture}

\theoremstyle{definition}

\newtheorem{prob}[equation]{Problem}
\newtheorem{rem}{Remark}[section]

\newcounter{minutes}
\divide\time by 60
\newcounter{hours}
\multiply\time by 60
\addtocounter{minutes}{-\time}

\newcounter {own}
\def\theown {\thesection       .\arabic{own}}

\newenvironment{pf}[1][]{%
 \vskip 3mm
 \noindent
 \ifthenelse{\equal{#1}{}}%
  {{\slshape Proof. }}%
  {{\slshape #1.} }%
 }%
{\qed\bigskip}

\newcounter{alphabet}
\newcounter{tmp}

\def\be{\begin{equation}}
\def\ee{\end{equation}}

\newcommand{\bee}{\begin{enumerate}}
\newcommand{\eee}{\end{enumerate}}

\newcommand{\blem}{\begin{lem}}
\newcommand{\elem}{\end{lem}}
\newcommand{\bthm}{\begin{thm}}
\newcommand{\ethm}{\end{thm}}
\newcommand{\bcor}{\begin{cor}}
\newcommand{\ecor}{\end{cor}}
\newcommand{\beg}{\begin{examp}}
\newcommand{\eeg}{\end{examp}}
\newcommand{\begs}{\begin{examples}}
\newcommand{\eegs}{\end{examples}}
\newcommand{\bdefe}{\begin{defin}}
\newcommand{\edefe}{\end{defin}}
\newcommand{\bprob}{\begin{prob}}
\newcommand{\eprob}{\end{prob}}
\newcommand{\bei}{\begin{itemize}}
\newcommand{\eei}{\end{itemize}}

\newcommand{\bcon}{\begin{conj}}
\newcommand{\econ}{\end{conj}}
\newcommand{\bcons}{\begin{conjs}}
\newcommand{\econs}{\end{conjs}}
\newcommand{\bprop}{\begin{prop}}
\newcommand{\eprop}{\end{prop}}
\newcommand{\br}{\begin{rem}}
\newcommand{\er}{\end{rem}}
\newcommand{\brs}{\begin{rems}}
\newcommand{\ers}{\end{rems}}
\newcommand{\bo}{\begin{obser}}
\newcommand{\eo}{\end{obser}}
\newcommand{\bos}{\begin{obsers}}
\newcommand{\eos}{\end{obsers}}
\newcommand{\bpf}{\begin{pf}}
\newcommand{\epf}{\end{pf}}
\newcommand{\ba}{\begin{array}}
\newcommand{\ea}{\end{array}}
\newcommand{\beq}{\begin{eqnarray}}
\newcommand{\beqq}{\begin{eqnarray*}}
\newcommand{\eeq}{\end{eqnarray}}
\newcommand{\eeqq}{\end{eqnarray*}}

\begin{document}

\title{On Coefficient Estimates of Negative Powers and Inverse Coefficients for Certain Starlike Functions}

\author{Md Firoz Ali}
\address{Md Firoz Ali,
Department of Mathematics,
Indian Institute of Technology Kharagpur,
Kharagpur-721 302, West Bengal, India.}
\email{ali.firoz89@gmail.com}

\author{A. Vasudevarao}
\address{A. Vasudevarao,
Department of Mathematics,
Indian Institute of Technology Kharagpur,
Kharagpur-721 302, West Bengal, India.}
\email{alluvasu@maths.iitkgp.ernet.in}


\def\thefootnote{}
\footnotetext{ {\tiny File:~\jobname.tex,
printed: \number\year-\number\month-\number\day,
          \thehours.\ifnum\theminutes<10{0}\fi\theminutes }
} \makeatletter\def\thefootnote{\@arabic\c@footnote}\makeatother

\begin{abstract}
For $-1\le B<A\le 1$, let $\mathcal{S}^*(A,B)$ denote the class of normalized analytic functions $f(z)= z+\sum_{n=2}^{\infty}a_n z^n$ in $|z|<1$ which satisfy the subordination relation $zf'(z)/f(z)\prec (1+Az)/(1+Bz)$ and $\Sigma^*(A,B)$ be the corresponding class of meromorphic functions in $|z|>1$. For $f\in\mathcal{S}^*(A,B)$ and $\lambda>0$, we shall estimate the absolute value of the Taylor coefficients $a_n(-\lambda,f)$ of the analytic function $(f(z)/z)^{-\lambda}$. Using this we shall determine the coefficient estimate for inverses of functions in the classes $\mathcal{S}^*(A,B)$ and $\Sigma^*(A,B)$.
\end{abstract}

\thanks{}

\maketitle
\pagestyle{myheadings}
\markboth{Md Firoz Ali and A. Vasudevarao }{Coefficient Estimates for Negative Powers and Inverses of Functions}

\section{Introduction}

Let $\mathcal{A}$ denote the family of analytic functions $f$ in the unit disk $\mathbb{D}:=\{z\in\mathbb{C}:\, |z|<1\}$ of the form
\begin{equation}\label{p3_i001}
f(z)= z+\sum_{n=2}^{\infty}a_n z^n.
\end{equation}
Let $\mathcal{S}$ denote the class of univalent (one-to-one) functions in $\mathcal{A}$.
For $-1\le B<A\le 1$, let $\mathcal{S}^*(A,B)$ denote the class of functions $f\in\mathcal{A}$ satisfying
\begin{equation}\label{p3_i005}
\frac{z f'(z)}{f(z)} \prec \frac{1+A z}{1+B z} \quad\mbox{ for } z\in\mathbb{D},
\end{equation}
where $\prec$ stands for the usual subordination for analytic functions in $\mathbb{D}$. This class was introduced by Janowski \cite{Janowski-1973} and extensively studied in the literature (see \cite{Goel-Mehrok-1981,Silverman-Silvia-1985} and references there in). For $-1\le B<A\le 1$, functions in $S^*(A,B)$ are known to be in $\mathcal{S}$. Let $\mathbb{\widetilde{C}}=\mathbb{C}\cup\{\infty\}$ and $\Sigma$ denote the class of meromorphic univalent functions $g$ defined in $\Delta:=\{z\in\mathbb{\widetilde{C}}: |z|>1\}$ of the form
\begin{equation}\label{p3_i010}
g(z)= z\left(1+\sum_{n=1}^{\infty}b_nz^{-n}\right) \quad \mbox{ for } 1<|z|<\infty.
\end{equation}
For $-1\le B<A\le 1$, let $\Sigma^*(A,B)$ denote the class of functions $g\in\Sigma$ which satisfy
$$
\frac{1}{z}\frac{ g'(1/z)}{g(1/z)} \prec \frac{1+A z}{1+B z} \quad\mbox{ for } z\in\mathbb{D}.
$$
The reason for choosing this typical notation will be discussed at the end of this section. The mapping $f(z)\mapsto g(z):=1/f(1/z)$ establishes a one-to-one correspondence between functions in the classes $\mathcal{S}$ and $\Sigma$ and also between functions in the classes $\mathcal{S}^*(A,B)$ and $\Sigma^*(A,B)$ because
$$
\frac{zg'(z)}{g(z)}= \frac{z\left(1/f(1/z)\right)'}{1/f(1/z)} = \frac{1}{z}\frac{ f'(1/z)}{f(1/z)} \quad\mbox{ for } |z|>1.
$$
For suitable choice of the parameters $A$ and $B$, we can obtain different subclasses studied by various authors. For instance, we list some of the subclasses for certain parameters:
\begin{enumerate}

\item $\mathcal{S}^* :=\mathcal{S}^*(1,-1)$ is the class of starlike functions and $\Sigma^*:=\Sigma^*(1,-1)$ is the corresponding class of meromorphic starlike functions.

\item For $0\le\alpha < 1$, $\mathcal{S}^*(\alpha):=\mathcal{S}^*(1-2\alpha,-1)$ is the class of starlike functions of order $\alpha$ and $\Sigma^*(\alpha):=\Sigma^*(1-2\alpha,-1)$ is the corresponding class of meromorphic starlike functions of order $\alpha$.

\item The class $\mathcal{S}^*(1,0)$ was introduced by Singh \cite{Singh-1968}.

\item The class $\mathcal{S}(\alpha):=\mathcal{S}^*(\alpha,-\alpha)$ for $0<\alpha\le 1$ was introduced by Padmanabhan \cite{Padmanabhan-1968}.

\item The class $\mathcal{S}^*(1,-1+1/\alpha)$ for $\alpha \ge 1/2$ was introduced by Singh et al. \cite{Singh-Singh-1974}.

\item The class $\mathcal{S}^*((b^2-a^2+a)/b, (1-a)/b)$, for $a+b \ge 1$, $b \le a\le 1+b$ was introduced by Silverman \cite{Silverman-1978}.

\end{enumerate}

For $f\in\mathcal{S}$ and $\lambda\in\mathbb{R}$, the function $(f(z)/z)^{-\lambda}$ is analytic in $\mathbb{D}$ and has the representation of the form
\begin{equation}\label{p3_i015}
\left(\frac{f(z)}{z}\right)^{-\lambda} = 1+ \sum_{n=1}^{\infty} a_n(-\lambda,f)z^n \quad\mbox{ for } z\in\mathbb{D}.
\end{equation}
One of the well-known extremal problems in the theory of univalent functions is to estimate the modulus of the Taylor coefficients $a_n(-\lambda,f)$ given by (\ref{p3_i015}). This problem has been extensively studied in the literature.  As a consequence of L. de Branges's proof of the celebrated Bieberbach conjecture (see \cite{Branges-1985}), Hayman and Hummel \cite{Hayman-Hummel-1986} proved that for every $f\in\mathcal{S}$ and $\lambda\le -1$,
\begin{equation}\label{p3_i020}
|a_n(-\lambda,f)| \le |a_n(-\lambda,k)|=
\left|\left(\begin{array}{c}
2\lambda\\
n
\end{array}
\right)\right|
\end{equation}
holds for every $n\ge 1$ where $k(z)=z/(1-z)^2$ is the Koebe function. But if $\lambda\in(-1,0)$ then (\ref{p3_i020}) is no longer true for $n=2$ (see \cite{Hayman-Hummel-1986}) which can be shown by a theorem of Fekete and Szeg\H{o} \cite{Fekete-Szego-1933}. More generally, Grinsphan \cite{Grinsphan-1983} showed that (\ref{p3_i020}) is false for all even coefficients, if $-1<\lambda<0$ and for all odd coefficients, if $-1/2<\lambda<0$.

The case of negative powers was first considered by L. de Branges \cite{Branges-1986} and emphasized the importance of this case in connection with the
general coefficient problem for univalent functions. By means of L\"{o}wner theory, Roth and Wirths \cite{Roth-Wirths-2001} has shown that for every $f\in\mathcal{S}$ and $\lambda> 0$ the estimate (\ref{p3_i020}) still holds for $1\le n\le \lambda+1$ except for the classical inequality $|a_2(1,f)|= |a_2(-1,f)- a_1^2(-1,f)|\le 1$.

If $F$ is the inverse of a function $f\in\mathcal{S}$ then $F$ has the following representation
\begin{equation}\label{p3_i025}
F(w)= w+ \sum_{n=2}^{\infty} A_n w^n
\end{equation}
which is valid in some neighborhood of the origin. Again, if $g\in\Sigma$ is of the form (\ref{p3_i010}) then its inverse $G(w)$ has the following representation
\begin{equation}\label{p3_i030}
G(w)= w\left(1+\sum_{n=1}^{\infty} B_nw^{-n}\right)
\end{equation}
which is valid in some neighborhood of infinity. Similarly, for $-1\le B<A\le 1$ if $f\in \mathcal{S}^*(A,B)$ is of the form (\ref{p3_i001}) then $f(z)$ has inverse $F(w)$ of the form (\ref{p3_i025}). Also, if $g\in\Sigma^*(A,B)$ is of the form (\ref{p3_i010}) then it has inverse $G(w)$ of the form (\ref{p3_i030}).

Determination of the sharp coefficient estimates of inverse functions in various subclasses of the class of analytic and univalent functions is an interesting problem in geometric function theory. In 1923, L\"{o}wner \cite{Lowner-1923} proved that if $F(w)= w+ \sum_{n=2}^{\infty} A_n w^n$ is the inverse of a function $f\in \mathcal{S}$ then
$$
|A_n|\le \frac{2n!}{n!(n+1)!} \quad\mbox{ for } n\ge 2
$$
and the inequality is sharp for the inverse of the Koebe function $k(z)= z/(1-z)^2$. An alternative approach to the inverse coefficient problem for functions in the class $\mathcal{S}$ has been investigated by Schaeffer and Spencer \cite{Schaeffer-Spencer-1945} and FitzGerald \cite{FitzGerald-1972}. Although, the inverse coefficient problem for the class $\mathcal{S}$ has been completely solved in 1923, only a few complete results are known on inverse coefficient estimates for most of the subclasses of the class $\mathcal{S}$ (for instance, see \cite{Juneja-Rajasekaran-1953, Krzyz-Libera-Zlotkiewicz-1979, Libera-Zlotkiewicz-1984, Libera-Zlotkiewicz-1985,Libera-Zlotkiewicz-1992, Livingston-1984, Silverman-1989}). Even, in some cases the inverse coefficients showed unexpected behavior. For example, it is known that if $f\in\mathcal{C}$, the class of convex univalent functions in the unit disk $\mathbb{D}$ of the form (\ref{p3_i001}) then the coefficients of its inverse function satisfies $|A_n|\le 1$ for $n=1,2,\ldots,8$ (see \cite{Libera-Zlotkiewicz-1982}); while $|A_{10}|>1$ (see \cite{Kirwan-Schober-1979}) and the exact bounds of $|A_9|$ and $|A_n|$ for $n>10$ are still unknown. In 1979, Krzy\.{z} et al. \cite{Krzyz-Libera-Zlotkiewicz-1979} determined the sharp inverse coefficient estimate of $|A_2|$ and $|A_3|$ for functions in the class $\mathcal{S}^*(\alpha)$. Further these results have been extended by Kapoor and Mishra \cite{Kapoor-Mishra-2007} (see also \cite{Srivastava-Mishra-Kund-2011}).

Let $\widetilde{\Sigma}$ denote the class of meromorphic univalent functions $h$ defined on the punctured unit disk $\mathbb{D}^*:=\{z\in\mathbb{C}: 0<|z|<1\}$ of the form $h(z)= \frac{1}{z}\left(1+\sum_{n=1}^{\infty}c_nz^{n}\right)$.
For $-1\le B<A\le 1$, let $\widetilde{\Sigma}^*(A,B)$ denote the class of functions $h\in\widetilde{\Sigma}$ which satisfy
$$
-\frac{z h'(z)}{h(z)} \prec \frac{1+A z}{1+B z} \quad\mbox{ for } z\in\mathbb{D}.
$$
Then the mapping $f(z)\mapsto h(z):=1/f(z)$ establishes a one-to-one correspondence between functions in the classes $\mathcal{S}$ and $\widetilde{\Sigma}$ and also between functions in the classes $\mathcal{S}^*(A,B)$ and $\widetilde{\Sigma}^*(A,B)$. Moreover, the mapping $g(z)\mapsto h(z):=g(1/z)$ establishes a one-to-one correspondence between functions in the classes $\Sigma$ and $\widetilde{\Sigma}$, and also between functions in the classes $\Sigma^*(A,B)$ and $\widetilde{\Sigma}^*(A,B)$. Indeed, the functions $g\in\Sigma$ and $h(z)=g(1/z)\in\widetilde{\Sigma}$ have the same geometric properties. Therefore it is sufficient to study one of the classes $\Sigma^*(A,B)$ or $\widetilde{\Sigma}^*(A,B)$. But this is not the case if we concern with coefficients of inverse functions of functions from these classes. This fact can be illustrated as follows:

Let $f(z) =z/(1-z)$ for $z\in\mathbb{D}$. It is easy to verify that $f\in\mathcal{S}^*$. Then $g(z)=1/f(1/z)=z-1$ belongs to the class $\Sigma^*:=\Sigma^*(1,-1)\subseteq\Sigma$ and $h(z)=1/f(z)=\frac{1}{z}-1$ belongs to the class $\widetilde{\Sigma}^*=\widetilde{\Sigma}^*(1,-1)\subseteq\widetilde{\Sigma}$. On the other hand, if $G(w)$ and $H(w)$ are the inverses of $g(z)$ and $h(z)$ respectively then $G(w)=w+1$ and $H(w)=(1+w)^{-1}=1-w+w^2-\cdots$.
In general, if $g$ is of the form (\ref{p3_i010}) and belongs to $\Sigma$ or $\Sigma^*(A,B)$ then it has the inverse $G(w)$ of the form (\ref{p3_i030}) but no such analog holds for $h$ in $\widetilde{\Sigma}$ or $\widetilde{\Sigma}^*(A,B)$.

In Section 2 we shall estimate the absolute value of the coefficients $a_n(-\lambda,f)$ defined by (\ref{p3_i015}) for functions in $\mathcal{S}^*(A,B)$ and $\lambda>0$ and the same will be discussed for functions in the class $\mathcal{S}^*$. As a simple application of this result, the problem of finding the sharp coefficient estimates for functions in the class $\Sigma^*(A,B)$ will completely been solved. In Section 3, we shall show that how the result obtained in Section 2 can be used to estimate the coefficients of inverse functions due to a result of Jabotinsky. Using this we shall determine the coefficient estimates of inverse functions of functions in the classes $\mathcal{S}^*(A,B)$ and $\Sigma^*(A,B)$ for $-1\le B<A\le1$.

\section{Coefficient Estimates of Negative Powers}

Throughout in this paper we assume that $-1\le B<A\le 1$ and $\delta=(1-A)/(1-B)$ unless otherwise stated. It is easy to see that $0\le\delta<1$.
For $\lambda>1$, partition the interval $0\le\delta<1$ into semi-closed intervals
\begin{align}\label{p3_m001}
\mathbb{J}(\lambda)=\left\{ \delta\in\mathbb{R}: 0\le\delta< \frac{\lambda-[\lambda]}{\lambda} \right\}
\end{align}
and for $k=0,1,\ldots, [\lambda]-1,$
\begin{align}\label{p3_m005}
\mathbb{I}_k(\lambda)=\left\{ \delta\in\mathbb{R}: \frac{\lambda-[\lambda]+k}{\lambda}\le\delta< \frac{\lambda-[\lambda]+k+1}{\lambda} \right\},
\end{align}
where $[\lambda]$ denotes the greatest integer less than or equal to $\lambda$. Note that $\mathbb{J}(\lambda)=\emptyset$ if $\lambda$ is a positive integer.
Define the function $k_{A,B}(z)$ by
\begin{equation}\label{p3_m010}
k_{A,B}(z) = \left\{\begin{array}{ll}
\displaystyle z e^{Az} & \mbox{ for } B=0\\[4mm]
z(1+Bz)^{\frac{A}{B}-1} &\mbox{ for } B\ne 0
\end{array} \right.
\end{equation}
and the function $k_{A,B}^{(n)}(z)$ by
\begin{equation}\label{p3_m015}
k_{A,B}^{(n)}(z) = (k_{A,B}(z^n))^{1/n}\quad\mbox{ for }n\in\mathbb{N}.
\end{equation}
It is easy to see that $k_{A,B}(z)$ and $k_{A,B}^{(n)}(z)$ belong to the class $\mathcal{S}^*(A,B)$.
In the following lemma we provide the essential ingredients which will be required to prove our main result in this section.

\begin{lem}\label{p3-lemma-001}
Let $-1\le B<A\le 1$ and $\lambda\in\mathbb{R}$ be fixed. Then for any $l\in\mathbb{N}$
\begin{align}\label{p3_m020}
\lambda^2(A-B)^2 &+ \sum_{n=1}^{l-1} \left[(\lambda(A-B)+nB)^2-n^2\right] \prod_{j=0}^{n-1}\left(\frac{\lambda(A-B)+Bj}{j+1}\right)^2\\
&= \frac{1}{((l-1)!)^2} \prod_{j=0}^{l-1} \left(\lambda (A-B)+Bj\right)^2.\nonumber
\end{align}
\end{lem}

\begin{proof}
For $l=1$ the identity (\ref{p3_m020}) is trivially true. Suppose that (\ref{p3_m020}) holds for $l=k-1,\, k\ge 2$. Then
a simple computation gives (with $l=k$ in LHS of (\ref{p3_m020}))
\begin{equation*}
\begin{split}
& \lambda^2(A-B)^2  + \sum_{n=1}^{k-1} \left[(\lambda(A-B)+nB)^2-n^2\right] \prod_{j=0}^{n-1}\left(\frac{\lambda(A-B)+Bj}{j+1}\right)^2 \\
&\quad = \lambda^2(A-B)^2 + \sum_{n=1}^{k-2} \left[(\lambda(A-B)+nB)^2-n^2\right] \prod_{j=0}^{n-1}\left(\frac{\lambda(A-B)+Bj}{j+1}\right)^2\\
&\quad\quad + \left[(\lambda(A-B)+B(k-1))^2-(k-1)^2\right] \prod_{j=0}^{k-2}\left(\frac{\lambda(A-B)+Bj}{j+1}\right)^2\\
&\quad = \frac{1}{((k-2)!)^2} \prod_{j=0}^{k-2}\left(\lambda(A-B)+Bj\right)^2\\
&\quad\quad + \left[(\lambda(A-B)+B(k-1))^2-(k-1)^2\right] \prod_{j=0}^{k-2}\left(\frac{\lambda(A-B)+Bj}{j+1}\right)^2\\
&\quad = \frac{1}{((k-2)!)^2} \prod_{j=0}^{k-2}\left(\lambda(A-B)+Bj\right)^2 \left[1+ \frac{(\lambda(A-B)+B(k-1))^2-(k-1)^2}{(k-1)^2}\right]\\
&\quad = \frac{1}{((k-1)!)^2} \prod_{j=0}^{k-1}\left(\lambda(A-B)+Bj\right)^2.
\end{split}
\end{equation*}
Thus (\ref{p3_m020}) is true for $l=k$ and hence, by mathematical induction, the identity (\ref{p3_m020}) is true for every $l\in\mathbb{N}$.
\end{proof}

\begin{thm}\label{p3-theorem-001}
Let $f\in\mathcal{S}^*(A,B)$ be of the form (\ref{p3_i001}) and for each fixed $\lambda>0$, the function $(f(z)/z)^{-\lambda}$ is given by (\ref{p3_i015}). Also let $\delta=(1-A)/(1-B)$. Then\\
\begin{itemize}

\item[(i)] for $0<\lambda\le1$ and $0\le\delta<1$, we have
\begin{equation}\label{p3_m025}
\left|a_l(-\lambda,f)\right|\le \frac{\lambda(A-B)}{l}, \quad l= 1,2,\ldots;
\end{equation}

\item[(ii)] for $\lambda>1$, $\delta\in\mathbb{I}_k(\lambda)$, $k=0,1,\ldots, [\lambda]-1$ and $\lambda(1-\delta)\notin\mathbb{N}$ we have
\begin{align}
\left|a_l(-\lambda,f)\right| &\le \prod_{j=0}^{l-1} \frac{\lambda(A-B)+Bj}{j+1}, \quad l= 1,2,\ldots, [\lambda]-k \quad\mbox{ and }\label{p3_m030}\\
\left|a_l(-\lambda,f)\right| &\le \left(\frac{[\lambda]-k}{l}\right)  \prod_{j=0}^{[\lambda]-k-1} \frac{\lambda(A-B)+Bj}{j+1}, \quad l= [\lambda]-k+1,\ldots;\label{p3_m035}
\end{align}
while for $\lambda(1-\delta)\in\mathbb{N}$, (\ref{p3_m030}) holds for $l= 1,2,\ldots, [\lambda]-k+1$ and (\ref{p3_m035}) holds for $l= [\lambda]-k+2,\ldots$.\\
In particular, if $\delta\in\mathbb{I}_{[\lambda]-1}(\lambda)$ then
\begin{equation}\label{p3_m040}
\left|a_l(-\lambda,f)\right|\le \frac{\lambda(A-B)}{l}\quad \mbox{ for } l= 1,2,\ldots;
\end{equation}

\item[(iii)] for $\lambda>1$ and $\delta\in\mathbb{J}(\lambda)$, we have
\begin{align}
\left|a_l(-\lambda,f)\right| &\le \prod_{j=0}^{l-1} \frac{\lambda(A-B)+Bj}{j+1}, \quad l= 1,2,\ldots, [\lambda]+1, \quad\mbox{ and }\label{p3_m045}\\
\left|a_l(-\lambda,f)\right| &\le \left(\frac{[\lambda]-k}{l}\right)  \prod_{j=0}^{[\lambda]-k-1} \frac{\lambda(A-B)+Bj}{j+1}, \quad l= [\lambda]+2,\ldots.\label{p3_m050}
\end{align}

\end{itemize}
The estimates (\ref{p3_m025}), (\ref{p3_m030}), (\ref{p3_m040}) and (\ref{p3_m045}) are sharp.
\end{thm}

\begin{proof}

Let $f\in \mathcal{S}^*(A,B)$ and $g(z)=(z/f(z))^\lambda$. In view of the relation (\ref{p3_i005}), it follows that, there exists an analytic function $\omega:\mathbb{D}\rightarrow \overline{\mathbb{D}}$ such that
\begin{equation}\label{p3_m055}
\frac{z f'(z)}{f(z)} =1- \frac{zg'(z)}{\lambda g(z)}= \frac{1+Az \omega(z)}{1+Bz\omega(z)} \quad\mbox{ for } z\in\mathbb{D}
\end{equation}
which is equivalent to
\begin{equation}\label{p3_m060}
g'(z)= \omega(z)\left(\lambda(B-A)g(z)- Bzg'(z)\right).
\end{equation}
By substituting (\ref{p3_i015}) in (\ref{p3_m060}) and further simplification  shows that
\begin{align*}\label{p3_m065}
\sum_{n=1}^{l} n a_n(-\lambda,f)z^{n-1} &+ \sum_{n=l+1}^{\infty} c_n(-\lambda,f)z^{n-1}\\
&=  \omega(z) \left( \lambda(B-A)+ \sum_{n=1}^{l-1} \left(\lambda(B-A)-nB\right)a_n(-\lambda,f)z^n\right)
\end{align*}
for certain coefficients $c_n(-\lambda,f)$. By using Parseval-Gutzmer formula (see also Clunie's method \cite{Clunie-1959,Clunie-Koegh-1960,Robertson-1970}) we obtain
$$
\sum_{n=1}^{l} n^2 \left|a_n(-\lambda,f)\right|^2 \le \lambda^2(A-B)^2 + \sum_{n=1}^{l-1} \left(\lambda(B-A)-nB\right)^2 \left|a_n(-\lambda,f)\right|^2
$$
which can be written as
\begin{equation}\label{p3_m070}
l^2 \left|a_l(-\lambda,f)\right|^2 \le \lambda^2(A-B)^2 + \sum_{n=1}^{l-1} \left[\left(\lambda(A-B)+nB\right)^2-n^2\right] \left|a_n(-\lambda,f)\right|^2.
\end{equation}
Since $\left(\lambda(A-B)+nB\right)^2-n^2 = (1-B) \left(\lambda(A-B)+n(1+B)\right) \left(\lambda(1-\delta)-n\right)$,
where $\delta=(1-A)/(1-B)$, the sign of each term inside the summation symbol on the right-hand-side of (\ref{p3_m070}) depends on the expression $\lambda(1-\delta)-n$ for $n=1,2,\ldots, l-1$.

If $0<\lambda\le1$ then $\lambda(1-\delta)-n\le0$ for $n=1,2,\ldots, l-1$. Then from (\ref{p3_m070}) we obtain
$$
l^2 \left|a_l(-\lambda,f)\right|^2 \le \lambda^2(A-B)^2
$$
which gives the estimate (\ref{p3_m025}). To determine the sign of the expression $\lambda(1-\delta)-n$ for $\lambda>1$, we partition the interval $0\le\delta< 1$ into semi-closed sub-intervals $\mathbb{J}(\lambda)$ and $\mathbb{I}_k(\lambda)$, $k=0,1,\ldots [\lambda]-1$, given by (\ref{p3_m001}) and (\ref{p3_m005}), respectively.

For any fixed $k$ $(k=0,1,\ldots, [\lambda]-1)$, suppose that $\delta\in\mathbb{I}_k(\lambda)$. Then the following two cases arise.

\textbf{Case-1:} Let $\lambda(1-\delta)\notin\mathbb{N}$. In this case
\begin{equation}\label{p3_m075}
\left\{\begin{array}{ll}
\displaystyle \lambda(1-\delta)-n >0  & \mbox{ for } n= 1,2,\ldots, [\lambda]-k-1,\\[2mm]
\displaystyle \lambda(1-\delta)-n \le 0 & \mbox{ for } n= [\lambda]-k, [\lambda]-k+1,\ldots.
\end{array} \right.
\end{equation}
By considering only non-negative contribution on the right hand summation in (\ref{p3_m070}) and using (\ref{p3_m075}), it follows that
\begin{equation}\label{p3_m080}
l^2 \left|a_l(-\lambda,f)\right|^2 \le \lambda^2(A-B)^2 + \sum_{n=1}^{l-1} \left[\left(\lambda(A-B)+nB\right)^2-n^2\right] \left|a_n(-\lambda,f)\right|^2
\end{equation}
for $l=1,2,\ldots, [\lambda]-k$; while for $l= [\lambda]-k, [\lambda]-k+1,\ldots$, we have
\begin{equation}\label{p3_m085}
l^2 \left|a_l(-\lambda,f)\right|^2 \le \lambda^2(A-B)^2 + \sum_{n=1}^{[\lambda]-k-1} \left[\left(\lambda(A-B)+nB\right)^2-n^2\right] \left|a_n(-\lambda,f)\right|^2.
\end{equation}

We now apply the principle of mathematical induction on $l$. For $l=1$, it follows from (\ref{p3_m080}) that
$$
\left|a_1(-\lambda,f)\right| \le \lambda(A-B).
$$
This gives the estimate (\ref{p3_m030}) for $l=1$. For $l=2,\ldots, [\lambda]-k-1$, we now assume that
\begin{equation}\label{p3_m090}
\left|a_l(-\lambda,f)\right|\le \prod_{j=0}^{l-1} \frac{\lambda(A-B)+Bj}{j+1}
\end{equation}
holds true. From (\ref{p3_m080}), (\ref{p3_m090}) and Lemma \ref{p3-lemma-001}, it follows that
\begin{eqnarray*}
l^2 \left|a_l(-\lambda,f)\right|^2 & \le & \lambda^2(A-B)^2  + \sum_{n=1}^{l-1} \left[\left(\lambda(A-B)+nB\right)^2-n^2\right] \prod_{j=0}^{n-1} \left(\frac{\lambda(A-B)+B j}{j+1}\right)^2\\
&=& \frac{1}{((l-1)!)^2} \prod_{j=0}^{l-1}\left(\lambda(A-B)+Bj\right)^2.
\end{eqnarray*}
Therefore, for $l= 1,2,\ldots, [\lambda]-k$,
$$
\left|a_l(-\lambda,f)\right| \le \prod_{j=0}^{l-1} \frac{\lambda(A-B)+Bj}{j+1}
$$
which establishes the inequality (\ref{p3_m030}).

Further, if $l= [\lambda]-k+1,\ldots$ then by using (\ref{p3_m085}), the induction hypothesis (\ref{p3_m090}) and Lemma \ref{p3-lemma-001}, we obtain
\begin{eqnarray*}
l^2 \left|a_l(-\lambda,f)\right|^2 & \le & \lambda^2(A-B)^2 + \sum_{n=1}^{[\lambda]-k-1} \left[\left(\lambda(A-B)+ nB\right)^2-n^2\right] \prod_{j=0}^{n-1} \left(\frac{\lambda(A-B)+Bj}{j+1}\right)^2\\
&=& \frac{1}{(([\lambda]-k-1)!)^2} \prod_{j=0}^{[\lambda]-k-1}\left(\lambda(A-B)+Bj\right)^2
\end{eqnarray*}
or equivalently
$$
\left|a_l(-\lambda,f)\right| \le \frac{[\lambda]-k}{l} \prod_{j=0}^{[\lambda]-k-1} \frac{\lambda(A-B)+Bj}{j+1} \quad\mbox{ for } l= [\lambda]-k+1,\ldots
$$
which is precisely the inequality (\ref{p3_m035}).

\textbf{Case-2:} Let $\lambda(1-\delta)\in\mathbb{N}$. Note that the conditions $\delta\in\mathbb{I}_k(\lambda)$ $(k=0,1,\ldots, [\lambda]-1)$ and $\lambda(1-\delta)\in\mathbb{N}$ simultaneously imply that $\delta= (\lambda-[\lambda]+k)/\lambda$. In this case
\begin{equation*}
\left\{\begin{array}{ll}
\displaystyle \lambda(1-\delta)-n >0  & \mbox{ for } n= 1,2,\ldots, [\lambda]-k,\\[2mm]
\displaystyle \lambda(1-\delta)-n \le 0 & \mbox{ for } n= [\lambda]-k+1, [\lambda]-k+1,\ldots.
\end{array} \right.
\end{equation*}
Therefore by proceeding as in Case-1, one can see that (\ref{p3_m030}) holds for $l= 1,2,\ldots, [\lambda]-k+1$ and (\ref{p3_m035}) holds for $l= [\lambda]-k+2,\ldots$.

For $k=[\lambda]-1$, the two estimates in (\ref{p3_m030}) and (\ref{p3_m035}) respectively reduce to
\begin{equation}\label{p3_m095}
\left|a_1(-\lambda,f)\right|\le \lambda(A-B)
\end{equation}
and
\begin{equation}\label{p3_m100}
\left|a_l(-\lambda,f)\right|\le \frac{\lambda(A-B)}{l} \quad\mbox{ for } l=2,3,\ldots.
\end{equation}
Therefore (\ref{p3_m040}) follows by combining the  inequalities (\ref{p3_m095}) and (\ref{p3_m100}).

Finally, if $\lambda>1$ and $\delta\in\mathbb{J}(\lambda)$ then
\begin{equation*}
\left\{\begin{array}{ll}
\displaystyle \lambda(1-\delta)-n >0  & \mbox{ for } n= 1,2,\ldots, [\lambda]\\
\displaystyle \lambda(1-\delta)-n \le 0 & \mbox{ for } n= [\lambda]+1, [\lambda]+2,\ldots.
\end{array} \right.
\end{equation*}
Therefore by proceeding as in Case-1, one can obtain the estimates (\ref{p3_m045}) and (\ref{p3_m050}).

Equality holds in (\ref{p3_m030}) and (\ref{p3_m045}) for the function $k_{A,B}(z)$ defined by (\ref{p3_m010}). On the other hand equality holds in (\ref{p3_m025}) and (\ref{p3_m040}) for every $l\in\mathbb{N}$ for the function $k_{A,B}^{(l)}(z)$ defined by (\ref{p3_m015}).

\end{proof}

The following special cases of Theorem \ref{p3-theorem-001} may be of worth mentioning.

\begin{cor}\label{p3-corollary-001}
Let $f\in\mathcal{S}^*$ be of the form (\ref{p3_i001}) and for each fixed $\lambda>0$, the function $(f(z)/z)^{-\lambda}$ is given by (\ref{p3_i015}). Then\\
\begin{itemize}

\item[(i)] for $0<\lambda\le1$, we have
\begin{equation}\label{p3_m105}
\left|a_l(-\lambda,f)\right|\le \frac{\lambda(A-B)}{l}, \quad l= 1,2,\ldots;
\end{equation}

\item[(ii)] for $\lambda>1$, we have
\begin{align}
\left|a_l(-\lambda,f)\right| &\le \prod_{j=0}^{l-1} \frac{2\lambda-j}{j+1}, \quad l= 1,2,\ldots, [\lambda]+1\quad\mbox{ and }\label{p3_m110}\\
\left|a_l(-\lambda,f)\right| &\le \left(\frac{[\lambda]-k}{l}\right)  \prod_{j=0}^{[\lambda]-k-1} \frac{2\lambda-j}{j+1}, \quad l= [\lambda]+2,\ldots;\label{p3_m115}
\end{align}
\end{itemize}
The inequalities (\ref{p3_m105}) and (\ref{p3_m110}) are sharp.
\end{cor}

\begin{cor}\label{p3-corollary-005}
Let $f\in\mathcal{S}^*(A,B)$ be of the form (\ref{p3_i001}) for some $-1\le B<A\le 1$. Also, let
\begin{equation}\label{p3_m120}
\frac{z}{f(z)} = 1+ \sum_{n=1}^{\infty} a_n(-1,f)z^n \quad\mbox{ for } z\in\mathbb{D}.
\end{equation}
Then for $n\ge 1$ we have
\begin{equation}\label{p3_m125}
\left|a_n(-1,f)\right|\le \frac{A-B}{n}.
\end{equation}
The inequality (\ref{p3_m125}) is sharp and equality is attained for the function $k_{A,B}^{(n)}(z)$ defined by (\ref{p3_m015}).
\end{cor}

Let $g\in\Sigma^*(A,B)$ be given by (\ref{p3_i010}). Due to the one-to-one correspondence between functions in the classes $\mathcal{S}^*(A,B)$ and $\Sigma^*(A,B)$, there exists a unique function $f\in\mathcal{S}^*(A,B)$ such that $g(z)=1/f(1/z)$. From (\ref{p3_i010}) and (\ref{p3_m120}) one can easily see that $b_n=a_n(-1,f)$ for $n\ge 1$. In view of Theorem \ref{p3-theorem-001}, we obtain $|b_n|\le (A-B)/n$ for $n\ge 1$ which leads to the result of Karunakaran \cite[Theorem 1]{Karunakaran-1976} except for the changes in the notation.

\section{Inverse Coefficient Estimates}

In this section our main concern is to find the coefficient estimates for inverses of functions in the classes $\mathcal{S}^*(A,B)$ and $\Sigma^*(A,B)$. The key idea is to link the Taylor coefficients of $(f(z)/z)^{-\lambda}$ to the Taylor coefficients of $(F(w)/w)^{-\mu}$ where $F(w)$ denotes the inverse function of $f(z)$. Indeed, the following result fulfills our requirement which is a reformulation of a result due to Jabotinsky \cite[Theorem II]{Jabotinsky-1953}.

\begin{lem}\cite[Theorem II]{Jabotinsky-1953}\label{p3-lemma-005}
Let the function $f\in \mathcal{A}$ be of the form (\ref{p3_i001}). Then the inverse function $F(w)$ of the function $f(z)$ is analytic in $|w|<\rho$ for some $\rho>0$. Also suppose that
$$
\left(\frac{z}{f(z)}\right)^{t} = 1+ \sum_{n=1}^{\infty} a_n(-t,f)z^n
$$
and
$$
\left(\frac{w}{F(w)}\right)^{t} = 1+ \sum_{n=1}^{\infty} A_n(-t,F)w^n
$$
where $t= \pm1, \pm2, \pm3, \ldots$. Then
$$
A_n(t,F)= \frac{t}{t+n}\,a_n(-(t+n),f) \quad\mbox{ for }~~  t+n \ne0 \quad\mbox{ and }~~ t= \pm1, \pm2, \pm3, \ldots
$$
and $A_{-t}(t,F)$ is given by
\begin{equation}\label{p3_m130}
\sum_{t=-\infty}^{\infty} A_{-t}(t,F)z^{-t-1} = \frac{f'(z)}{f(z)}.
\end{equation}
\end{lem}

The following theorem gives the coefficient estimates for the inverse of functions in the class $\mathcal{S}^*(A,B)$.

\begin{thm}\label{p3-theorem-005}
Let $f\in\mathcal{S}^*(A,B)$ be of form(\ref{p3_i001}) and $F(w)$ be its inverse which is valid in some neighborhood of the origin and has the representation $F(w)= w+ \sum_{n=2}^{\infty} A_n w^n$. Also let $\delta=(1-A)/(1-B)$ and for a fixed $n\in\mathbb{N}$, $\mathbb{I}_k(n)=\left[\frac{k}{n},\frac{k+1}{n}\right)$, $k=0,1,\ldots, n-1$. Then\\
\begin{enumerate}
\item [(i)] for $\delta\in \mathbb{I}_0(n) \bigcup \mathbb{I}_1(n)$, we have
\begin{equation}\label{p3_m135}
\left|A_n\right|\le \frac{1}{n}\prod_{j=0}^{n-2} \frac{n(A-B)+Bj}{j+1};
\end{equation}

\item [(ii)] for $\delta\in \mathbb{I}_k(n)$, $k=2,3,\ldots, n-2$, we have
\begin{equation}\label{p3_m140}
\left|A_n\right|\le \frac{n-k}{n(n-1)}\prod_{j=0}^{n-k-1} \frac{n(A-B)+Bj}{j+1};
\end{equation}

\item [(iii)] for $\delta\in \mathbb{I}_{n-1}(n)$, we have
\begin{equation}\label{p3_m145}
\left|A_n\right|\le \frac{A-B}{n-1}.
\end{equation}
\end{enumerate}
The estimates (\ref{p3_m135}) and (\ref{p3_m145}) are sharp.

\end{thm}

\begin{proof}
It is well-known that (see e.g. \cite[Vol-I, p. 54]{Goodman-book-1983})
$$
A_n = \frac{1}{2\pi in} \int_{|z|=r} \frac{1}{[f(z)]^n}\, dz = \frac{1}{n}\, a_{n-1}(-n,f)
$$
for $n\ge 2$, where $a_{n-1}(-n,f)$ is defined by (\ref{p3_i015}). In order to estimate $|A_n|$ we shall estimate $|a_{n-1}(-n,f)|$ using Theorem \ref{p3-theorem-001}. We note that the inequality (\ref{p3_m030}) is applicable only for $k=0$ and $k=1$ (in case of $n(1-\delta)\in\mathbb{N}$ the inequality (\ref{p3_m030}) is also applicable for $k=2$). Therefore for $\delta\in \mathbb{I}_0(n) \bigcup \mathbb{I}_1(n)$, the inequality (\ref{p3_m030}) yields
$$
\left|A_n\right|= \frac{1}{n}\left|a_{n-1}(-n,f)\right| \le \frac{1}{n}\prod_{j=0}^{n-2} \frac{n(A-B)+Bj}{j+1}
$$
which is precisely the inequality (\ref{p3_m135}). Similarly, for $\delta\in \mathbb{I}_k(n)$, $k=2,3,\ldots, n-2$, the inequality (\ref{p3_m035}) yields,
$$
\left|A_n\right|= \frac{1}{n}\left|a_{n-1}(-n,f)\right| \le \frac{n-k}{n(n-1)}\prod_{j=0}^{n-k-1} \frac{n(A-B)+Bj}{j+1}
$$
which is precisely the inequality (\ref{p3_m140}). Finally, for $\delta\in \mathbb{I}_{n-1}(n)$, the inequality (\ref{p3_m040}) gives
$$
\left|A_n\right|= \frac{1}{n}\left|a_{n-1}(-n,f)\right| \le \frac{A-B}{n-1}
$$
which is precisely the inequality (\ref{p3_m145}).

Equality holds in the estimate (\ref{p3_m135}) for the function $k_{A,B}(z)$ defined by (\ref{p3_m010}) and equality holds in the estimate (\ref{p3_m145}) for the function $k_{A,B}^{(n-1)}(z)$ defined by (\ref{p3_m015}).

\end{proof}

\begin{rem}
As we mentioned in the proof of Theorem \ref{p3-theorem-005} that the inequality (\ref{p3_m030}) is also applicable for $k=2$ when $n(1-\delta)\in\mathbb{N}$ to estimate $|a_{n-1}(-n,f)|$. Note that the condition $k=2$ and $n(1-\delta)\in\mathbb{N}$ imply that $\delta=2/n$. Hence for $\delta=2/n$, the estimate (\ref{p3_m135}) holds and equality holds for the function $k_{A,B}(z)$. Note that for $\delta=2/n$, the estimates (\ref{p3_m135}) and (\ref{p3_m140}) give the same bound for $|A_n|$.
\end{rem}

The following special cases of Theorem \ref{p3-theorem-005} may be of worth mentioning.
\begin{cor}\label{p3-corollary-015}
Let $f\in\mathcal{S}^*(1,B)$ for some $-1\le B<  1$ and $F(w)$ be its inverse which is valid in some neighborhood of the origin and has the representation $F(w)= w+ \sum_{n=2}^{\infty} A_n w^n.$ Then
\begin{equation}\label{p3_m150}
\left|A_n\right|\le \frac{1}{n}\prod_{j=0}^{n-2} \frac{n(1-B)+Bj}{j+1}.
\end{equation}
The inequality (\ref{p3_m150}) is sharp and equality holds for the function $k_{1,B}(z)$ defined by (\ref{p3_m010}).

\end{cor}

If we choose $B=-1$ in Corollary \ref{p3-corollary-015} we obtain the coefficient estimate for inverses of starlike functions in the class $\mathcal{S}^*$ which has been proved by L\"{o}wner \cite{Lowner-1923} (see also \cite{FitzGerald-1972, Schaeffer-Spencer-1945}). Again, if we choose $B=0$ in Corollary \ref{p3-corollary-015} then we obtain the coefficient estimate for inverses of functions in the class $\mathcal{S}^*(1,0)$. Moreover, if we substitute $B=-1+1/\alpha$ with $\alpha > 1/2$ in Corollary \ref{p3-corollary-015} then we obtain the coefficient estimate for inverses of functions in the class $\mathcal{S}^*(1,-1+1/\alpha)$.
Finally, if we choose $A=1-2\alpha$ and $B=-1$ with $0\le\alpha < 1$ in Theorem \ref{p3-theorem-005} then Theorem \ref{p3-theorem-005} reduces to \cite[Theorem 1] {Kapoor-Mishra-2007} which solves the inverse coefficient problem for the class $\mathcal{S}^*(\alpha)$.

The following theorem gives the coefficient estimates for inverse functions of functions in the class $\Sigma^*(A,B)$.

\begin{thm}\label{p3-theorem-010}
Let $g\in\Sigma^*(A,B)$ be of the form (\ref{p3_i010}) and $G(w)$ be its inverse which is valid in some neighborhood of infinity and has the representation $G(w)= w\left(1+\sum_{n=1}^{\infty} B_nw^{-n}\right)$. Also let $\delta=(1-A)/(1-B)$ and for a fixed $n\in\mathbb{N}$, $\mathbb{I}_k(n)=[\frac{k}{n},\frac{k+1}{n})$, $k=0,1,\ldots, n-1$. Then\\
\begin{enumerate}
\item [(i)] for $-1\le B< A \le 1$ we have
\begin{equation}\label{p3_m155}
\left|B_1\right|\le A-B \quad\mbox{ and }\quad \left|B_2\right|\le (A-B)/2;
\end{equation}

\item [(ii)] for $n\ge 2$ and $\delta\in \mathbb{I}_k(n)$, $k=0,1,\ldots, n-2$ we have
\begin{equation}\label{p3_m160}
\left|B_{n+1}\right|\le \frac{n-k}{n(n+1)}\prod_{j=0}^{n-k-1} \frac{n(A-B)+Bj}{j+1};
\end{equation}

\item [(iii)] for $n\ge 2$ and $\delta\in \mathbb{I}_{n-1}(n)$ we have
\begin{equation}\label{p3_m165}
\left|B_{n+1}\right|\le \frac{A-B}{n+1}.
\end{equation}
\end{enumerate}
The estimates in (\ref{p3_m155}) and (\ref{p3_m165}) are sharp.
\end{thm}

\begin{proof}
Let $g\in\Sigma^*(A,B)$. Then there exists a unique function $f\in\mathcal{S}^*(A,B)$ such that $g(z)=1/f(1/z)$. Also, it can be easily verified that $G(w)=1/F(1/w)$, where $F(w)$ is the inverse of $f(z)$. Therefore,
\begin{equation}\label{p3_m170}
B_n= A_n(-1,F) \quad \mbox{ for } n\ge 1,
\end{equation}
where $A_n(-1,F)$'s are defined as in Lemma \ref{p3-lemma-005}.
Our first aim is to find the upper bound for $B_1$. Since $f\in\mathcal{S}^*(A,B)$ and $f(z)$ is of the form (\ref{p3_i001}), it is easy to see that
\begin{equation}\label{p3_m175}
\frac{f'(z)}{f(z)}= \frac{1}{z} + a_2+ (2a_3-a_2^2)z +\cdots.
\end{equation}
By comparing the coefficients in (\ref{p3_m130}) and (\ref{p3_m175}), we obtain $A_1(-1,F) = a_2$. Therefore by using the well-known estimate $|a_2|\le A-B$ (see for example \cite[Theorem 1]{Goel-Mehrok-1981}) we obtain
$$
|B_1|= |A_1(-1,F)|=|a_2| \le A-B.
$$

Our next aim is to find the upper bound for $A_n(-1,F)$ for $n\ge2$. From Lemma \ref{p3-lemma-005}, it is easy to see that
\begin{equation}\label{p3_m180}
A_n(-1,F) = - \frac{1}{n-1} a_n(-(n-1),f) \quad \mbox{ for } n\ge2.
\end{equation}
Therefore from (\ref{p3_m170}) and (\ref{p3_m180}) we obtain
\begin{equation}\label{p3_m185}
|B_{n+1}|= \left|A_{n+1}(-1,F)\right|= \frac{1}{n}\left|a_{n+1}(-n,f)\right| \quad\mbox{ for } n\ge 1.
\end{equation}
We note that the inequality (\ref{p3_m030}) in Theorem \ref{p3-theorem-001} is not applicable for any value of $k=0,1,\ldots, n-1$ to estimate $|a_{n+1}(-n,f)|$. Hence from (\ref{p3_m185}) and the inequality (\ref{p3_m035}) it is easy to see that
$$
|B_{n+1}|= \frac{1}{n}\left|a_{n+1}(-n,f)\right| \le \frac{n-k}{n(n+1)}\prod_{j=0}^{n-k-1} \frac{n(A-B)+Bj}{j+1}
$$
for $n\ge2$ and $\delta\in \mathbb{I}_k(n)$, $k=0,1,\ldots, n-2$. Again, from (\ref{p3_m040}) and (\ref{p3_m185}) we obtain
\begin{equation}\label{p3_m190}
|B_{n+1}|= \frac{1}{n}\left|a_{n+1}(-n,f)\right| \le  \frac{A-B}{n+1}
\end{equation}
for $n\ge 1$ and $\delta\in\mathbb{I}_{n-1}(n)$. For $n=1$, the inequality (\ref{p3_m190}) gives the estimate of $|B_2|$ which is precisely the second inequality of (\ref{p3_m155}).

Equality holds in (\ref{p3_m155}) for the function $1/k_{A,B}(1/z)$,  where $k_{A,B}(z)$ is defined by (\ref{p3_m010}) and equality holds in (\ref{p3_m165}) for the function $1/k_{A,B}^{(n)}(1/z)$, $n\ge 2$, where $k_{A,B}^{(n)}(z)$ is defined by (\ref{p3_m015}).

\end{proof}

\noindent\textbf{Acknowledgement:}  The authors thank Prof. S. Ponnusamy for useful discussions and  careful reading the paper. The authors also thank him for bringing the paper \cite{Karunakaran-1976} to their attention. The first author thank University Grants Commission for the financial support through UGC-SRF Fellowship. The second author thank SRIC, IIT Kharagpur for the support.

\end{document}